\documentclass[runningheads]{llncs}

\usepackage{graphicx}
\usepackage{xcolor}
\usepackage{hyperref}
\usepackage{xspace}
\usepackage{amssymb,amsmath,amsfonts,graphicx,psfig}

\def\aiOrcid{\hspace*{.4mm}\includegraphics[scale=.5]{orcid_16x16.png}}
\definecolor{orcidlogocol}{HTML}{A6CE39}
\def\orcidID#1{\href{https://orcid.org/#1}{\textcolor{orcidlogocol}{\aiOrcid} }}

\definecolor{darkblue}{rgb}{0.0, 0.0, 0}
\newcommand{\blue}[1]{{#1}\xspace}
\newcommand{\darkblue}[1]{{\color{darkblue} #1}\xspace}

\newcommand{\opt}{\ensuremath{\operatorname{\textsc{Opt}}}\xspace}

\newcommand{\Reg}{Regular}
\newcommand{\reg}{regular}

\begin{document}
\title{Approximation Algorithms for Graph Burning}
%
%
\author{Anthony Bonato\inst{1}
\and
Shahin Kamali\inst{2}
}
\authorrunning{A. Bonato and S. Kamali}
%
\institute{Ryerson University, Toronto, ON, Canada \\ \email{\href{mailto:abonato@ryerson.ca}{abonato@ryerson.ca}}
\and University of Manitoba, Winnipeg, MB, Canada  \\
\email{ \href{mailto:shahin.kamali@umanitoba.ca}{shahin.kamali@umanitoba.ca} }}
\maketitle              
\begin{abstract}
Numerous approaches study the vulnerability of networks against social contagion. Graph burning studies how fast a contagion, modeled as a set of fires, spreads in a graph. The burning process takes place in synchronous, discrete rounds. In each round, a fire breaks out at a vertex, and the fire spreads to all vertices that are adjacent to a burning vertex. The selection of vertices where fires start defines a schedule that indicates the number of rounds required to burn all vertices. Given a graph, the objective of an algorithm is to find a schedule that minimizes the number of rounds to burn graph. Finding the optimal schedule is known to be NP-hard, and the problem remains NP-hard when the graph is a tree or a set of disjoint paths. The only known algorithm is an approximation algorithm for disjoint paths, which has an approximation ratio of 1.5.
\vspace{0.1in}

We present approximation algorithms for graph burning. For general graphs, we introduce an algorithm with an approximation ratio of 3. When the graph is a tree, we present another algorithm with approximation ratio 2. Moreover, we consider a setting where the graph is a forest of disjoint paths. In this setting, when the number of paths is constant, we provide an optimal algorithm which runs in polynomial time. When the number of paths is more than a constant, we provide two approximation schemes: first, under a regularity condition where paths have asymptotically equal lengths, we show the problem admits an approximation scheme which is fully polynomial. Second, for a general setting where the regularity condition does not necessarily hold, we provide another approximation scheme which runs in time polynomial in the size of the graph.

\end{abstract}

\section{Introduction}
%
Numerous efforts were initiated to characterize and analyze social contagion or social influence in networks; see, for example, \cite{BondFJKMSF12,Fajardo2013,Kramer12,Kramer14}.
These studies investigate the vulnerabilities and strengths of these networks against the spread of an emotional state or other data, such as a meme or gossip. For example, there are studies that suggest emotional states can be transferred to others via emotional contagion on Facebook; such emotional contagion is known to occur without direct interaction between people and in the complete absence of nonverbal cues \cite{Kramer14}.

The \emph{burning number} \cite{6BonatoJR14,5BonatoJR16} measures how prone a network is to fast social contagion. In the burning protocol, like many other network protocols, data is communicated between nodes in discrete rounds. The input is an undirected, unweighted, finite simple graph. We say a node is burning if it has received data. Initially, no vertex is burning. In each round, a burning vertex sends data to all its neighbors, and all neighbors will be on fire at the end of the round; this is consistent with the fact that a user in the network can expose all its neighbours to a posted piece of data. In addition, in each given round, a new fire starts at a non-burning vertex called an \emph{activator}; this can be interpreted as a way to target additional users that initiate the contagion. Note that the burning protocol does not provide a specified algorithm of how the fire spreads. However, the algorithm can choose where to initiate the fire. The decisions of the algorithm for the location of activators define a \emph{schedule} that can be described by a \emph{burning sequence}: the $i$th member of the burning sequence indicates the vertex at which a fire is started in round $i$.
We say the graph is \emph{burned} when all vertices are on fire; that is, all members of the network have received the data.
Figure~\ref{fig:burn-process} provides an illustration of the burning process.

To understand how prone a graph is to the spread of data, we are interested in schedules that minimize the number of rounds required to burn the whole graph. The {burning number} of a given graph is the minimum such number; hence, an optimal algorithm burns the graph in a number of rounds that is equal to the burning number. 
Unfortunately, finding optimal solutions is NP-hard even for elementary graph families \cite{4BessyBJRR17}. The focus of this paper is to provide approximation algorithms for burning graphs.

\begin{figure}
\centering
\includegraphics[scale=.6,trim={5cm 20cm 5cm 14cm}, clip, scale=.66]{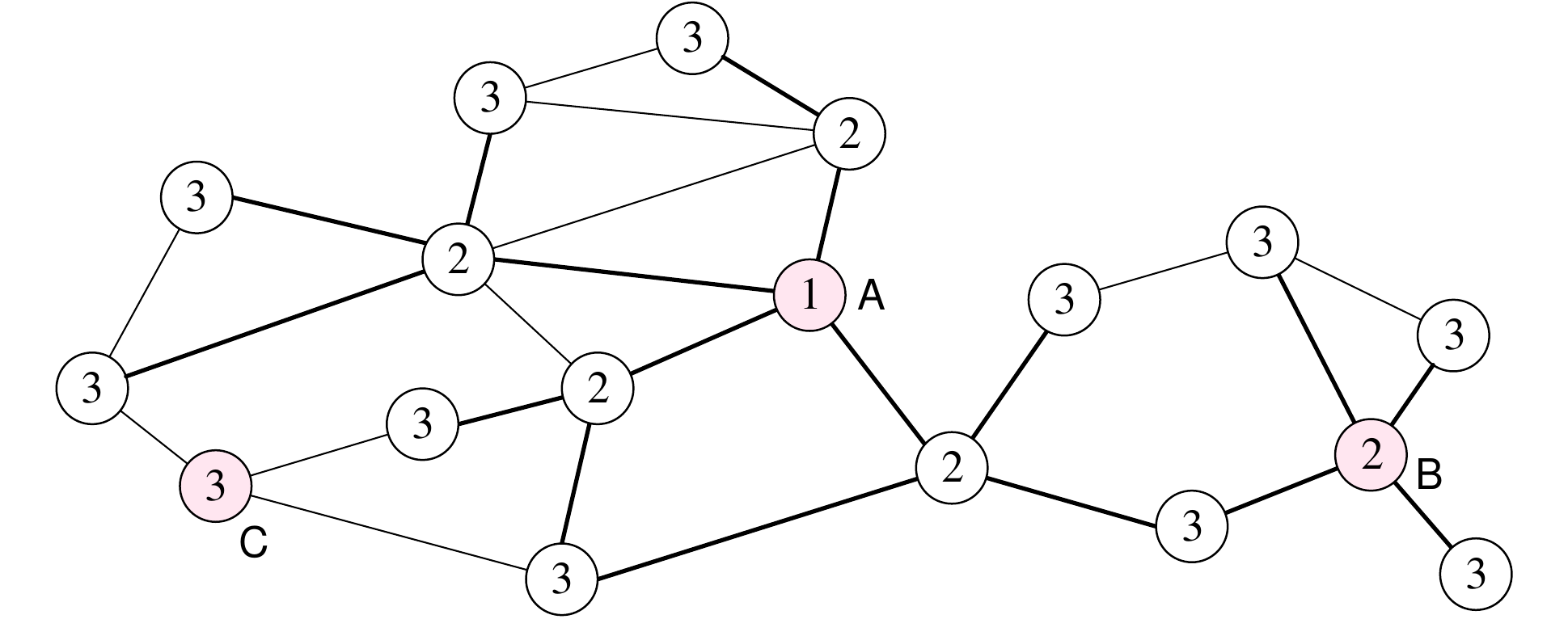}
\caption{Burning a graph in three rounds using a schedule defined by burning sequence $\langle A,B,C \rangle$. The number \blue{on each vertex} indicates the rounds at which \blue{the vertex becomes a burning vertex}. At round 1, a fire starts at $A$. At round 2, another fire starts at $B$ while the fire at $A$ spreads to all neighbors of $A$. At round 3, the fire spreads to all vertices except for $C$, where a new fire is started.}\label{fig:burn-process}
\end{figure}

\subsection*{Previous work}
Bonato et al.\ \cite{6BonatoJR14,5BonatoJR16} first introduced the burning process as a way to model spread of contagion in a social network; they characterized the burning number for some graph classes and proved some properties for the burning number. The results of \cite{1MitschePR17,product18} extended these results for additional graph families and also studied a variant of burning number in which the burning sequence is selected according to some probabilistic rule. Bessy et al.\ \cite{3BessyBJRR18} further studied the burning number and proved that for a connected graph of size $n$ the burning number is at most $2\lceil\sqrt{n}\rceil-1$ and conjectured that this number is indeed at most $\lceil \sqrt{n} \rceil$. They proved better bounds for the burning number of trees. Land and Lu \cite{LandL16} slightly improved the upper bound to $\frac{\sqrt{6}}{2} \sqrt{n}$. Sim et al.\ \cite{Paterson} provided tight bounds for the burning number of generalized Petersen graphs.
Bonato et al.\ \cite{4BessyBJRR17} proved that it is NP-hard to find a schedule that completes burning in the minimum number of rounds (in time equal to the burning number). Interestingly, their hardness result holds for basic graph families such as acyclic graphs with maximum degree three, spider graphs, and path forests (that is, a disjoint union of paths).

There are numerous gossiping and broadcasting protocols that \blue{aim to} model the amount of time it takes to spread information throughout a given network. 
For example, in the \emph{telephone model} for gossiping, there is a distinguished originator that starts spreading the gossip. 
In a given round, each node that has received a \blue{piece of data (gossip)} can inform one of its neighbors via a phone call. A gossip schedule defines the order in which each node informs its neighbors. The goal of a schedule is to minimize the number of rounds required to inform all vertices.  This problem is known to be NP-hard \cite{GareyJ79,SlaterCH81} (in fact, APX-hard \cite{Schindelhauer00}) and there is an approximation algorithm completes within a sublogarithmic factor of optimal schedule \cite{ElkinK06} (whether a constant approximation algorithm exists is an open problem). We refer the reader to \cite{HedetniemiHL88,NikzadR14,Ravi94} for more results on telephone broadcasting. It is evident that the telephone model is not suitable for 
situations where a user can expose all its neighbors by posting a gossip and without in-person communication with them. The \emph{Radio model} is more relevant in this context, where each informed vertex broadcasts the message to all its neighbors; however, in this model, there is a pre-defined set of originators and it is often assumed that vertices have limited information about graph structures 
(see, for example, \cite{CzumajR06,GhaffariHK15,KowalskiP07,Peleg07}).

Social contagion is important from a \emph{viral marketing} perspective, based on the observation that targeting a small set of users can have a cascading word-of-mouth effect in a social network. Domingos and Richardson \cite{DomingosR01,RichardsonD02} define influence maximization problems that aim to define a set of initially activated user that can eventually influence a maximum members of the network. This problem is known to be NP-hard. Kempe et al.\ provide several approximation algorithms for several simple diffusion models \cite{KempeKT03,KempeKT15} as well as a more general decreasing cascade model, where a behaviour spreads in a cascading fashion according to a probabilistic rule \cite{KempeKT05}. These results were followed by more approximation algorithms and inapproximability results for these models (see, for example, \cite{ChenGL11,ChenCCKLRSWWY11,ChenWY09}). We refer the reader to Kleinberg \cite{Kleinberg13} for the economic aspects of cascading behaviour on social networks. Note that besides the diffusion model, the influence maximization problem is different from burning in the sense that initial informed users start spreading data at the same time (while in burning they start one at a time).

Another problem related to graph burning is the \emph{Firefighter Problem}, which also assumes discrete, synchronous rounds. Given a graph $G$, at round 1, a fire starts at a given node $r$ of $G$. In each subsequent round, a firefighter can defend one non-burning vertex while the fire spreads to all undefended neighbours of each burning vertex. Once burning or defended, a vertex remains so for all subsequent rounds. The process ends when the fire can no longer spread. The goal of an algorithm (which we identify with a firefighter) is to defend a maximum number of vertices that can be saved; that is, that are not burning at the end of the process. Despite similarities in the underlying model, the objective in the Firefighter problem is quite different from the burning problem. As expected, the Firefighter problem is NP-hard \cite{FinbowKMR07}, and it is known that no approximation algorithms can achieve a factor of $n^\alpha$ for any $\alpha<1$, assuming $P\neq NP$ \cite{AnshelevichCHS12}. The problem remains NP-hard for the trees \cite{FinbowKMR07}; however, there are constant-factor approximation algorithms for trees (see, for example, \cite{Hartnell,CaiVY08}).

\subsection*{Contributions}
The burning problem is NP-hard, which is not surprising as many related problems are NP-hard.  However, the fact that the problem remains NP-hard for elementary graph families such as path forests (that is, disjoint unions of paths) raises questions about its computational complexity. In particular, we may ask whether there is a polynomial algorithm that has a constant approximation ratio. Bonato and Lidbetter \cite{2017arXiv170709968B} answered this question for path forests in the affirmative by introducing a 3/2-approximation algorithm. The problem remained open for other graph families. This question is particularly interesting because it has different answers for similar problems (as described in the previous section): for telephone broadcasting, it remains open whether there is a constant approximation algorithm. For influence maximization, there is a constant approximation algorithm, while for the Firefighter problem, it is NP-hard to achieve a sublinear approximation ratio. 

In this paper, we show that there is indeed a simple polynomial algorithm with constant approximation ratio of at most $3$ for \emph{any} graph. Our algorithm is intuitive and runs in \blue{time} $O(m \log n)$ for a graph with $n$ vertices and $m$ edges. When the graph is a tree, we present another algorithm with improved approximation ratio of 2. 
Finally, we consider the problem when the graph is a path forest. 
In case the graph is formed by a constant number of paths, we present a dynamic programming algorithm that creates an optimal solution in polynomial time. When the number of paths is not a constant, we \blue{provide two approximation schemes. The first scheme works under a regularity condition which implies the lengths of paths are asymptotically equal. For this scheme, we}
reduce the problem to the bin covering problem to achieve a fully polynomial time approximation scheme (FPTAS) for the problem. \blue{For the general setting, when there is no assumption on the length of the paths, we use a different approach to present a polynomial time approximation scheme (PTAS) which runs in \blue{time} polynomial in the size of the graph.}
%


\section{Approximation Algorithm for Burning Graphs}
In this section, we devise an approximation algorithm with approximation factor of 3 for the burning problem. Throughout the section, we use $G=(V,E)$ to denote an input graph and \opt to denote the optimal algorithm for the problem. We use $\opt(G)$ to denote the burning number of $G$.
We begin with the following lemma.

%

\begin{lemma}\label{lem1}
For a positive integer $r$, if there are $r$ vertices 
at pairwise distance at least $2r-1$, then any burning schedule requires at least $r$ rounds to complete.
\end{lemma}
\begin{proof}
Let $x_1, x_2, \ldots, x_r$ be $r$ vertices of pairwise distance at least $2r-1$. For each $x_i$, consider the ball of radius $r-1$ formed by vertices of distance at most $r-1$ from $x_i$. Since the distance of $x_i$ and any $x_j$ is at least $2r-1$, their balls do not intersect. 
Assume there is a schedule that completes in at most $r-1$ rounds. That schedule should have a fire started inside each ball (a fire started at a distance $r$ or more reaches $x_i$ after at least $r$ rounds). Hence, at least $r$ fires must be started, which implies the burning completes in at least $r$ rounds. This contradicts the initial assumption that the schedule completes within $r-1$ rounds.
\end{proof}

We devise a procedure Burn-Guess$(G,g)$ that receives a `guess' value $g$ for the number of rounds required to burn graph $G$. The output of Burn-Guess is one of the following.

\begin{enumerate}
\item A schedule that completes burning in at most $3g-3$ rounds.
\item `Bad-Guess' that guarantees 
any schedule requires at least $g$ rounds to complete.
\end{enumerate}

To devise an approximation algorithm, it suffices to 
find the smallest guess value $g^*$ so that Burn-Guess$(g^*)$ returns a schedule (which implies Burn-Guess$(g^*-1)$ returns Bad-Guess). In this way, the returned schedule completes in at most $3g^*-3$ rounds while \opt requires at least $g^*-1$ rounds to complete. This results in an algorithm with approximation ratio of at most 3. 

Burn-Guess processes vertices one-by-one in an arbitrary order and 
maintains a set of `centers' that is initially empty. When processing a vertex $v$, the algorithm checks the distance of $v$ to its closest center. 
If such distance is at most $2g-2$, then $v$ is marked as `non-center'; otherwise, $v$ is added to the set of centers. 
In this way, all centers are at pairwise distance of at least $2g-1$. After processing any vertex, if the number of centers becomes equal to $g$, then Burn-Guess returns Bad-Guess. When all vertices are processed, the algorithm returns a schedule defined by a burning sequence formed by an arbitrary ordering of centers. 

%
%


\begin{lemma}\label{badguess}
If Burn-Guess$(G,g)$ returns Bad-Guess, then there is no burning schedule for $G$ that completes in less than $g$ number of rounds.
\end{lemma}

\begin{proof}
Burn-Guess returns Bad-Guess if the number of centers becomes equal to $g$. Since 
all centers are at pairwise distance of at least $2g-1$, we have that 
there are $g$ vertices at pairwise distance of $2g-1$ or more. Applying Lemma~\ref{lem1}, we conclude that any burning schedule requires at least $g$ rounds to burn the graph.
\end{proof}

\begin{lemma}\label{schedule}
If Burn-Guess$(G,g)$ returns a burning sequence, then the burning of that sequence completes in at most $3g-3$ rounds.
\end{lemma}
\begin{proof}
All non-center vertices are at distance at most $2g-2$ of at least one center. Recall that the burning schedule uses centers as activators. The fire starts at the last center at round $g-1$; all vertices within distance $2g-2$ of that center burn by round $3g-3$. 
We conclude that all non-center vertices burn by round $3g-3$. 
\end{proof}

We now arrive at our main result.

\begin{theorem}\label{mainresult}
There is a polynomial algorithm with approximation ratio of at most $3$ for burning any graph $G=(V,E)$.
\end{theorem}
\begin{proof}
Let $n= |V|$. The algorithm finds the smallest value $g^*$ 
for which Burn-Guess returns a schedule $(g^* \leq n)$. 
By Lemma \ref{schedule}, the schedule returned by Burn-Guess completes in at most $3g^*-3$ rounds. Meanwhile, since Burn-Guess returns Bad-Guess for $g^*-1$, be Lemma \ref{badguess}, no schedule completes in less than $g^*-1$ rounds. 
\end{proof}

%
It is not hard to see the upper bound in Theorem\ref{mainresult} is tight.
Consider the graph in Figure~\ref{fig:lowerSimpleAlg}, where $c_1, \ldots, c_k$ are the centers selected by the algorithm. 
Note the pairwise distance between any two centers is $2k$ and the distance of a non-center and a center is at most $2k-2$. So, Burn-Guess returns a schedule when its parameter is $g = k$ while it returns Bad-Guess when $g=k-1$. Assuming centers are burned in the same order, the cost of the algorithm is $3k-2$ (a fire starts at $c_k$ at round $k$ and reaches $b$ at round $3k-2$). 
On the other hand, there is a better scheme that burns vertex $a$ at round 1 and 
burns the middle point of the path $p$ between $c_k$ and $b$ at round 2. This scheme burns all vertices by round $k+1$. Consequently, the approximation ratio of the algorithm is at least $\frac{3k-2}{k+1}$ which converges to 3 for large values of $g$.

%

\begin{figure}[!h]
 \centering
\includegraphics[scale=.77,trim={2cm 11cm 12cm 22.4cm}, clip, scale=.7]{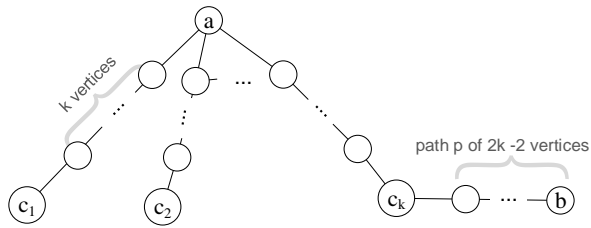}
\caption{An instance for which the scheme by the burning algorithm 
takes three times more rounds than the optimal algorithm to burn the graph.}\label{fig:lowerSimpleAlg}
\end{figure}

A straightforward implementation of the Burn-Guess uses breadth-first traversal of the graph. Starting with an unvisited node $v$, we add $v$ to the set of centers and apply breath first to visit all vertices within distance $2g-2$ of $v$. After reaching `depth' of $2g-2$, we stop the breath search and pick another unvisited vertex as the next center and start another breath first traversal. This process continues until all vertices are visited or the number of centers exceeds $g$. Clearly, any edge is visited at most once and hence Burn-Guess runs in \blue{time} $O(m)$. Since Burn-Guess is called $O(\log n)$ times (via a binary search in the space of $g$), we conclude that our algorithm for burning graphs runs in time $O(m \log n)$.

The above implementation is useful when the order in which vertices are processed is not defined by the algorithm. This is particularly handy when Burn-Guess has to work based on partial information; for example, a parallel setting where only a partition of the input graph is available to each processor. When there is no such restriction, we can apply optimizations like selecting the point located at the maximum distance to all current centers as the next center (this is similar to farthest-first algorithm for the metric $k$-center problem; see, for example, \cite{Vazir0004338}). While this optimization is likely to improve the approximation ratio \blue{(albeit with analysis techniques which would be more involved)} it degrades the running time: an efficient implementation of requires pre-computing all-pair shortest-path distances in $O(mn+n^2 \log n)$ using Dijkstra's algorithm. Provided with these distances, running an instance of Burn-Guess$(G,g)$ takes $O(n g)$, which is $O(n^2)$ for general graphs (and $O(n^{3/2})$ for connected graphs since $g\in O(\sqrt{n})$ when the graph is connected). Burn-Guess is called $O(\log n)$ times, which gives a total time complexity of $O(n^{2} \log n)$. This complexity is dominated by 
the $O(mn + n^2 \log n)$ of pre-computing pair-wise distances. 

\section{Approximation Algorithm for Trees}
In this section, we show that there is an algorithm with an approximation ratio of at most 2 for burning a tree $T$. In a way analogous to general graphs, the algorithm is based on a procedure Burn-Guess-Tree$(T,g)$ that guarantees the following for a given guess value $g$.
\begin{enumerate}
\item If the algorithm returns Bad-Guess, then any schedule for burning $T$ requires at least $g$ rounds to complete.
\item If the algorithm returns a schedule for burning $T$, then that schedule completes in at most $2g$ rounds.
\end{enumerate}

It is evident that, \blue{provided with the above guarantees,} the schedule returned for the smallest value of $g$ completes within twice the optimal schedule.

Given an input tree $T$, Burn-Guess-Tree$(T,g)$ selects an arbitrary node $s$ as the \emph{root} of the tree. The \emph{level} of a node $v$ is the distance of $v$ to $s$ and the \emph{$k$-ancestor} of $v$ is the vertex at distance $k$ from $v$ on its path to the root.
The procedure works in a number of steps and maintains a set of centers as well as a set of marked vertices. Initially, the set of centers is empty, and all vertices are unmarked.
At the beginning of a step $i$ (where $i\leq g$), the algorithm finds an unmarked vertex $v$ with the highest level. If the level of $v$ is at least $g$, then the $g$-ancestor of $v$ is added to the set of centers; otherwise, the root of $T$ is added to the set of centers. Meanwhile, all vertices within distance $g$ of the added center are marked. The procedure continues until all vertices are marked. In this case, the algorithm returns a burning sequence defined by an arbitrary ordering of centers as activators. If the number of centers becomes larger than $g$ before all vertices are marked, then the algorithm returns Bad-Guess. Figure \ref{fig:tree-process} illustrates the Burn-Guess-Tree procedure.

\begin{figure}[!b]
\centering
\includegraphics[scale=.67,trim={7cm 0cm 7cm 32.5cm}, clip, scale=.68]{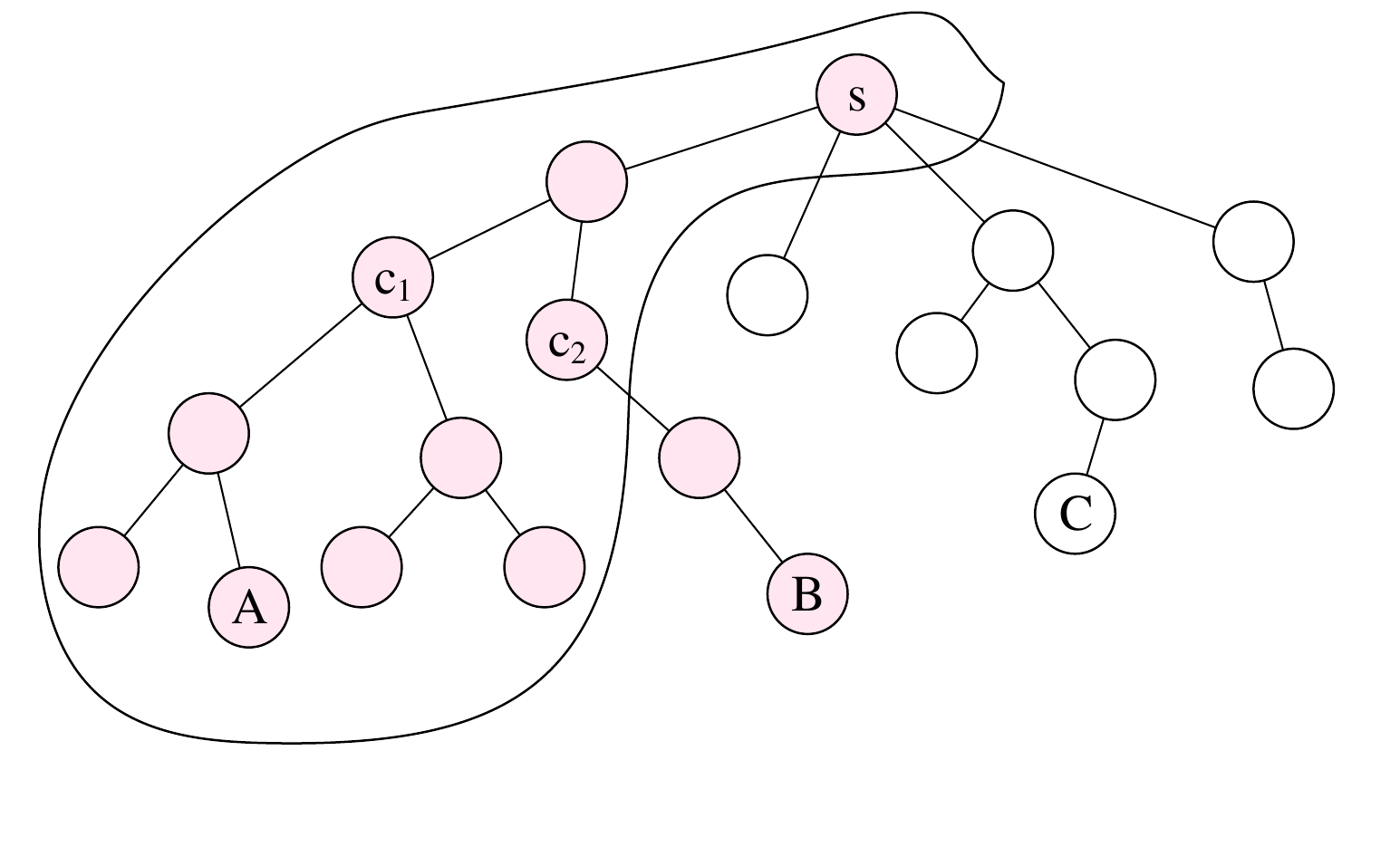}
\caption{An illustration of Burn-Guess-Tree with parameter $g=2$. The tree is rooted at $s$. The first selected vertex is $A$ and the first center is $c_1$. The next unmarked vertex with maximum level is $B$ and $c_2$ is selected as the next center. 
At this point, since there are still unmarked vertices (nodes that are not highlighted), the algorithm returns Bad-Guess. In the next iteration with $g=3$, the algorithm returns a schedule formed by the parent of $c_1$ and $s$.
}\label{fig:tree-process}
\end{figure}


\begin{lemma}\label{treeschedule}
If Burn-Guess-Tree$(T,g)$ returns a burning sequence, then the burning of that sequence completes in at most $2g$ rounds.
\end{lemma}

\begin{proof}
Since all vertices are marked, they are all within distance $g$ of a center. In the returned schedule, a fire is activated at all centers by round $g$, and consequently, all vertices are burned by round $2g$.
\end{proof}

Define a `$g$-site partition' as a set of at most $g$ vertices, called $g$ `sites', so that every vertex is within distance $g$ of its closest site.
We say that the tree admits the `$g$-site condition' if it has a $g$-site partition.
Clearly, in order to \blue{burn a tree} in less than $g$ rounds, the tree should pass the $g$-site condition; otherwise, any set of at most $g$ activators \blue{leaves} a vertex outside of \blue{combination of the spheres of all the activators} 
and hence, the burning process cannot complete within $g$ rounds.

\begin{lemma}\label{treebadg}
If Burn-Guess-Tree$(T,g)$ returns Bad-Guess, then $T$ does not admit $g$-site condition.
\end{lemma}

\begin{proof}

Consider otherwise; that is, there is a $g$-site partitioning $P$ defined by at most $g$ sites 
so that every vertex in the tree is within distance $g$ of one of the sites in $P$.
We show that it is possible to update $P$ so that it is still a $g$-site partitioning while having the set of $g$ activators selected by Burn-Guess-Tree$(T,g)$ as its set of sites. If that it is true, then every vertex is within distance $g$ of the $g$ centers selected by the algorithm.
However, we know at the time Bad-Guess was returned, there was an unmarked node at distance more than $g$ of the closest center (contradiction).

Let $c_1, c_2, \ldots, c_g$ be the centers selected by Burn-Guess-Tree (in the same order they are selected). We iteratively update $P$ by including these centers in its set of sites. At the beginning of iteration $i$, $P$ has $c_1, \ldots, c_{i-1}$ in it set of sites. In Burn-Guess-Tree, vertices at distance $g$ of these centers are marked. Let $v$ be the unmarked vertex with maximum distance from a marked node; so, \blue{following the definition of the Burn-Guess-Tree,} $c_i$ is the $g$-ancestor of $v$. Since $P$ is a $g$-site partitioning, it should have a site $c'$ within distance $g$ of $v$. Note that such site cannot be any of $c_j$ with $j<i$ since $v$ is unmarked. We argue that if $c'$ is replaced with $c_i$ in $S$, the partitioning still remains a $g$-site partitioning. For that, we show unmarked vertices within distance $g$ of $c'$ form a subset of unmarked vertices within distance $g$ of $c_i$. Consider otherwise; that is, there is an unmarked vertex $w$ at distance more than $g$ of $c_i$ and within distance $g$ of $c'$. If $w$ is in the tree rooted at $c_i$, then level of $w$ will be more than $v$ that contradicts $v$ being the unmarked vertex with the highest level. Next assume $w$ is outside of the tree rooted at $c_i$. Since $w$ has distance more than $g$ to $c_i$ and is within distance $g$ of $c'$, we conclude that $c'$ should be outside of the tree rooted at $c_i$; this contradicts $v$ being within distance $g$ of $c'$ (since $c_i$ is at distance $g$ of $v$). To summarize, after replacing $c'$ with $c_i$ in $P$, the partitioning remains a $g$-site partitioning. After repeating this process $g$ times, $P$ will be a $g$-site partitioning formed by the $g$ centers selected by Burn-Guess-Tree that is a contradiction as mentioned above.
\end{proof}

\begin{theorem}
There is a polynomial time algorithm with approximation ratio of at most 2 for burning a tree.
\end{theorem}

\begin{proof}
The algorithm finds the smallest value $g^*$ for which Broad-Guess-Tree returns a schedule. By Lemma \ref{treeschedule}, such a schedule burns the graph in at most $2g^*$ rounds. Meanwhile, since Burn-Guess-Tree returns Bad-Guess for $g^*-1$, by Lemma \ref{treebadg}, the tree does not have any $(g^*-1)$-site partition 
and hence, the cost of $\opt$ is more than $g^*-1$. In summary, the cost of the algorithm is at most $2g^*$, and the cost of the optimal solution is at least $g^*$.
\end{proof}

\section{Algorithms for Disjoint Paths}
Consider the burning problem when the input is a disjoint forest of paths. This problem is NP-hard and a 1.5 approximation exists for it \cite{2017arXiv170709968B}.
In this section, we present exact and approximation algorithms for this problem. When the graph is formed by a constant number of paths, we provide an exact algorithm that runs in polynomial time. For the more interesting case when the graph is formed by a non-constant number of paths,  
we provide \blue{two approximation} schemes for the problem. Throughout the section, we assume the input is a graph of size $n$ formed by $b$ paths of length $n_1,n_2, \ldots, n_b,$ where $n_i \leq n_{i+1}$.

\subsubsection*{Constant Number of Disjoint Paths}
We show that when the number of disjoint paths is constant, there is a polynomial time algorithm which provides optimal solution.
We call a graph $G$ an $r$-\emph{subset of graph} $G'$ if both $G$ and $G'$ are formed by $b$ disjoint paths of the same lengths, except for one path $p$ which has length $x$ in $G$ and length $x + i$ in $G'$ for some $i$ in the range $[0,2r+1]$. 

\begin{lemma}\label{dplemma}
A graph $G'$ formed by a forest of disjoint paths can be burned in $t$ rounds if and only if it has a $t$-subset $G$ which can be burned in $t-1$ rounds.
\end{lemma}

\begin{proof}
Assume $G'$ can be burned in $t$ rounds. Remove all vertices burned through the fire started at round 1. There will be at most $2t+1$ such vertices. Removing them will form a $t$-subset of $G$ and the same schedule can be used to burn that subgraph in $t-1$ rounds. Next, assume $G'$ has a $t$-subset $G$ which can be burned in $t-1$ rounds. To burn $G'$, we use the same schedule for burning $G$ except that at round 1 a fire is started at a node at distance $t$ of one endpoint of path $p$ which differentiates the two graphs. By round $t$, $2t+1$ vertices at distance $t$ of that node are burned. The remaining vertices that form $G$ can be burned in $t$ rounds following the same burning schedule for $G$. 
\end{proof}

The above lemma helps us devise a straightforward dynamic programming solution. \blue{We fill a table of size polynomial to $n$ (size of the graph) which has a boolean entry for each graph formed by $b$ paths of total size at most $n$ and for each deadline value $\tau$ (which is at most $n$). Such entry indicates whether the graph can be burned in $\tau$ rounds.  Using Lemma \ref{dplemma}, we can fill the table in a bottom-up approach. Additional \blue{bookkeeping} when filling the table leads us to the optimal burning scheme. 
}

\begin{theorem}
Given a graph of size $n$ formed by a forest of $b = \Theta(1)$ disjoint paths, there is an algorithm that generates an optimal burning scheme. The time complexity of the algorithm is polynomial in $n$.
\end{theorem}

\begin{proof}
Consider a dynamic-programming table $A$ of dimension $b+1$. Here, $A[t][x_1][x_2]\ldots[x_b]$ is a boolean value which indicates whether it is possible to burn a forest of $b$ paths within $t$ rounds, where the first path of the forest has length $x_1$, the second path has length $x_2$, and so on. In other words, an entry in the table is associated with a \emph{deadline} time $t$ (first dimension) and a graph formed by $b$ disjoint paths (subsequent $b$ dimensions). Note that the first dimension takes values between 1 and $n$ (the upper bound for burning time), while any other dimension takes values between 1 and $n_b$, where $n_b$ is the maximum length of any path. Consequently, the size of the table is $O(n^{b+1})$, which is polynomial in $n$ for constant values of $b$. To find the optimal burning time, after filling the table, find the smallest $t^*$ for which $A[t^*][n_1][n_2]\ldots[n_b]$ is True; recall that $n_1,\ldots,n_b$ are the lengths of paths in the input forest. Additional \blue{bookkeeping} when filling the table leads us to the optimal burning scheme which completes in $t^*$ rounds.

Next, we describe how to fill the table. Assume the table is processed (filled) for values up to $t-1$ for the first dimension; that is, the entries for graphs which can be burned within deadline $t-1$ are set to True. By Lemma~\ref{dplemma}, graphs with True entries for deadline $t$ have a $t$-subset with True entry for deadline $t-1$. Hence, for any entry with True value associated with deadline $t-1$ and graph $G'$, we set all entries associated with deadline $t$ and graphs having $G'$ as a $t$-subset to be True. In other words, if entry $A[t-1][x_1][x_2]\ldots[x_b]$ is true, then for any $j \leq b$, the entry $A[t][x_1][x_2]\ldots[x_j + i]\ldots x[b]$ will be set to True for $i\leq 2t-1$. In doing so, we also record the value of $j$. In this fashion, we can retrieve a burning schedule by looking at the index of the path at which a fire is started in each given round.
%
%
%
\end{proof}

\subsubsection*{FPTAS for Non-constant Number of \blue{\Reg} Disjoint Paths}
In this section, we use a reduction to the \emph{bin covering problem} to show the burning problem admits a \blue{\emph{fully} polynomial time approximation scheme (FPTAS)} when the input graph is formed by a non-constant number of \blue{`\reg'} disjoint paths. \blue{Here, we assume the paths are regular in the sense that the lengths of all paths are asymptotically equal.}
The bin covering problem is the dual to the classic bin packing problem and can be defined as follows.

\begin{definition}
The input to the bin covering problem is a multi-set of items with sizes in the range (0,1]. The goal is to `cover' a maximum number of size 1 with these items. By covering a bin, we mean assigning a multiset of items with total size at least 1 to the bin.
\end{definition}

Bin covering is NP-hard  \cite{Assman}; 
but there is an FPTAS for the problem:

\begin{lemma}\cite{JansenS03}\label{Jensen}
There is an algorithm A that, given a multiset $L$ of $n$ items with sizes $s(a_i)\in
(0, 1]$ and a positive number $\epsilon_0 > 0$, produces a bin covering of $L$ such that $A(L) \geq (1-\epsilon_0) OPT(L)$, assuming $OPT(L)$ is sufficiently large. The time complexity of A is polynomial in $n$ and $1/\epsilon_0$.
\end{lemma}

To provide an FPTAS for the burning problem on \blue{regular} disjoint paths, we reduce the problem to the bin covering problem. Before presenting the reduction, we state two lemmas with respect to the bin covering problem:

\begin{lemma}\label{twobins}
Assume two bins $B_1$ and $B_2$ are covered with a multiset of items so that $B_1$ only includes items of sizes at most $1/3$ and $B_2$ includes two items of size at least $2/3$. It is possible to modify the covering so that each bin has an item of size at least $2/3$ and both bins are still covered.
\end{lemma}

\begin{proof}
\blue{Since} all items in $B_1$ have size at most 1/3, it is possible to select a subset $S$ of these items which has total size in $(1/3,2/3]$ (start with an empty $S$ and repeatedly add items until the total size is in the desired range). Move items of $S$ from $B_1$ to $B_2$ and move an item of size at least $2/3$ from $B_2$ to $B_1$. Both bins will be covered in the result and each contain an item of size at least $2/3$.
\end{proof}

\begin{lemma}\label{optcover}
If we remove a multiset of total size $x$ from an instance of bin covering, then the number of covered bins in the optimal packing reduces by at least $x/2$.
\end{lemma}

\begin{proof}
If we remove a multiset of items with total size at least 1, then the number of covered bins decreases by at least 1. Otherwise, if removing a set of items with total size at least 1 does not reduce the number of covered bins, then these items can cover a new bin without impacting coverage of other bins. This contradicts the optimality assumption for the covering.
Given as multiset of total size $x$, we can partition it into $x/2$ multisets of total size in $[1,2)$ (this is possible because all items have size at most 1). Repeating the above argument $x/2$ times completes the proof.
\end{proof}

Consider an instance $I$ of the graph burning problem formed by $b$ paths $P_1, \ldots, P_b$ of lengths $n_1,n_2, \ldots, n_b$ such that $n_i \leq n_{i+1}$. 
Let $m_i = \lceil (n_i+1)/2\rceil$ and $C = 3 m_b$.
 We define the \emph{$k$-instance of bin covering associated with $I$} as an instance of bin covering formed by $b$ `large' items $\{p_1, \ldots, p_b\}$, where $p_i$ has size $1-m_i/C$ for $1\leq i \leq b$. We also define $k$ `small' items $\{q_1, \ldots, q_k\}$, where $q_j$ has size $\min\{j/C,1/3\}$ for $1\leq j \leq k$. 
Note that all large items have size at least $2/3$ and small items have size at most $1/3$.
Also note that large items appear in same way for any value of $k$ in the $k$-instances of the bin covering problem.
Figure \ref{fig:red} illustrates this construction.
\blue{Since paths are regular, 
the size large items is upper-bounded by a constant $c^*=1-m_1/(3m_b)$ which we refer to as the \emph{canonical constant} of the graph burning instance.}
\blue{Intuitively, burning the $b$ disjoint paths is translated to covering $b$ bins.
By Lemma \ref{twobins}, the $b$ large items can be placed in distinct bins without changing the number of covered bins. 
The remaining space of bins (to be covered) translates to paths of different length that should be burned. Small items are associated with the radii of the fires started at different rounds. These intuitions are formalized in the following two lemmas:  }

\begin{lemma}\label{SSS}
Given a solution for the $k$-instance of bin covering that covers at least $b$ bins, one can find, in polynomial time, a burning scheme that completes in at most $k$ rounds.
\end{lemma}
\begin{proof}
Given the solution for bin covering, we apply Lemma \ref{twobins} to ensure that there are $b$ bins that each include exactly one large item (this is possible because large items have size at least 2/3 and small items have size at most 1/3). Call the resulting bins $B_1, \ldots, B_b$, where $B_i$ is the bin that includes the large item $p_i$. Let $S_i$ be the set of small items in $B_i$. We associate items in $S_i$ with activators in a burning schedule. Assume initially all vertices are unmarked. We process small items in the solution for bin covering in the following manner. If $q_j$ $(1 \leq j \leq k)$ appears in set $S_i$ in the covering solution, then at time $k - j$ we start a fire at distance $j$ of the left-most marked node in path $P_i$ and mark any node at distance $j$ of it. In this wya, by the end of round $k$ all marked nodes will be burned. Since the total size of items in $S_i$ is at leas $m_i/C$, the number of marked vertices by the time $k$ would be $2m_i+1 \geq n_i$; that is, all vertices will be burned by the end of round $k$.
\end{proof}

\begin{figure}
\centering
\includegraphics[scale=.65,trim={0cm 33cm 0cm 0cm}, clip, scale=.68]{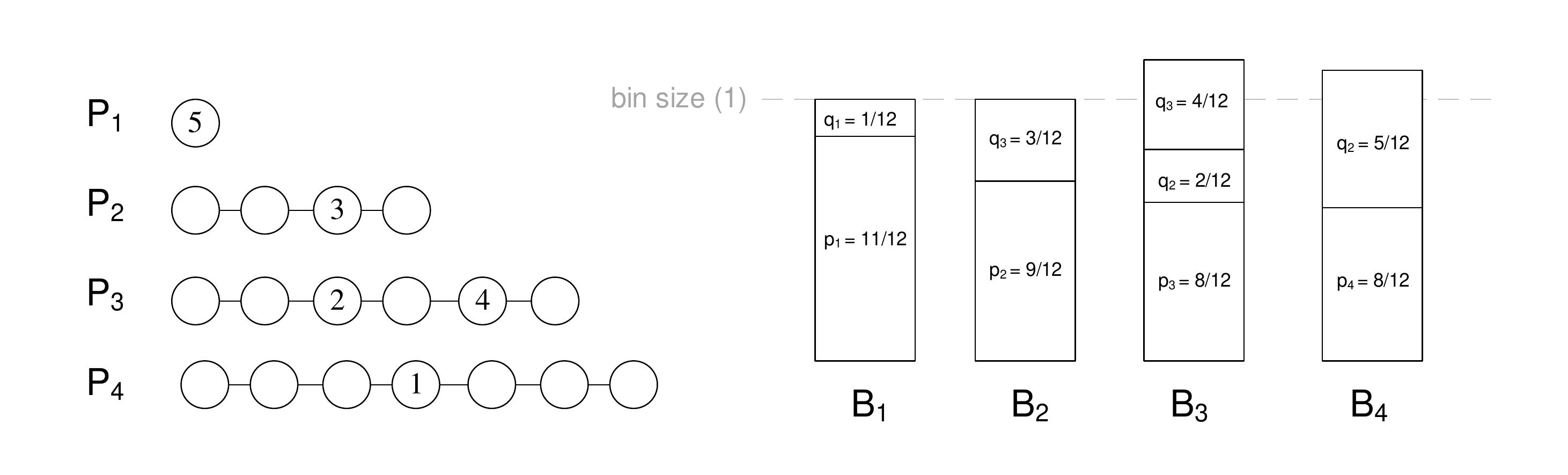}
\caption{A burning scheme for an instance of the burning problem on disjoint paths (left) and the equivalent covering for the $5$-instance of the covering problem (right). Here, we have $m_1=1, m_2 =3, m_3=4$, and $m_4=4$. }\label{fig:red}
\end{figure}

\begin{lemma}\label{TTT}
Given a burning scheme that completes within $k-1$ rounds, it is possible to create a solution for the $k$-instance of bin covering, in polynomial time, so that at least $b$ bins are covered in the solution.
\end{lemma}

\begin{proof}
We say an edge is burned if its both endpoints are burned. Since the burning scheme completes in at most $k-1$ rounds, we can burn all edges within $k$ rounds even if the lengths of all paths is increased by 2. 

Consider a path $P_i$ of length $n_i$. Assume the burning schedule starts fires at rounds     $k-y_1, k-y_2, \ldots, k-y_t$ in $P_i$. Note that a fire started at round $x$ burns at most $2(k-x)$ edges within $k$ rounds. Hence, if $Y$ denotes the total sum of $y_j$'s, then at most $2Y$ edges are burned by round $k$. Since all edges can be burned within $k$ rounds even in a longer path of $n_i+2$ vertices, we have that $2Y \geq n_i + 1$; that is $Y \geq m_i$.

We create a solution for the covering problem as follows. Place the large items in separate bins, and let $B_i$ be the large bin at which $p_i$ is placed. Recall that fires  at the path $P_i$ are started at rounds $k-y_1, \ldots, k-y_t$.
Consider the set $Q_i= \{\min\{y_1/C,1/3\}, \ldots, \min\{y_t/C,1/3\}\}$, which is a subset of small items in the covering instance. 
We place items in $Q_i$ in the bin that contains large item $p_i$. Next, we show the total size of items in the bin $B_i$ is at least 1. First, note that if any item in $Q_i$ has size 1/3, since $p_i$ has size at least 2/3, the total size of these two items will be 1 and we are done. Next, assume all items in $S$ are smaller than 1/3; that is, $Q_i = \{ y_1/C, \ldots, y_t/C\}$. The total size of items in $Q_i$ is equal to $Y/C$ which is at least $m_i/C$. So, the total size of items in the bin will be at least $m_i/C$ (for small items) plus $1-m_i/C$ (for the large item $b_i$) which sums to at least 1. In summary, for any $i\leq k$, if we can burn edges in path $P_i$ within $k$ rounds, we can cover the bin $B_i$ with small items associated with the rounds at which $P_i$ is burned. \end{proof}


We repeatedly apply the FPTAS of Lemma \ref{Jensen} (with a carefully chosen value of $\epsilon_0$) to find the smallest $k$ such that, for the $k$-instance of bin covering, the FPTAS returns a solution that covers at least $b$ bins.  By Lemma~\ref{SSS}, such solution can be converted to a burning scheme. Using Lemmas \ref{TTT},\ref{optcover}, we can show that this solution runs achieves approximation ratio of $1+\epsilon$ while running in time polynomial in both $n$ and $1/\epsilon$. More formally, we can prove the following theorem:

\begin{theorem}
Given a graph of size $n$ formed by a forest of $b = \omega(1)$ regular disjoint paths and a positive value $\epsilon$, there is an algorithm that generates a burning scheme that completes within a factor 
$1+\epsilon$
of an optimal scheme. The time complexity of the algorithm is polynomial in both $n$ and $1/\epsilon$.
\end{theorem}

\begin{proof}
Define $\epsilon_0 = \frac{(1-c^*)\epsilon}{4+(5-c^*)\epsilon}$ (recall that $c^*$ is the canonical constant of the regular instance of the burning problem).
Find the smallest $k$ such that for the $k$-instance of bin covering, the FPTAS of Lemma \ref{Jensen} \blue{with parameter $\epsilon_0$} returns a solution that covers at least $b$ bins. Let that value of $k$ be $k^*$. By Lemma~\ref{SSS}, that solution can be converted, in polynomial time, to a burning scheme that completes in $k^*$ rounds. 
Note that the total size of small items in the $k^*$-instance is $k^*(k^*+1)/(2C)$, and the total size of large items is at most $bc^*$. 
Since $b$ bins are covered, we conclude that $k^*(k^*+1)/(2C) \geq b- bc^*$; that is $\frac{bC}{k^*(k^*+1)}\leq \frac{1}{2(1-c^*)}$.
This implies that for large values of $k$ we have $\frac{bc}{k'^2} <\frac{1}{1-c^*}$ where $k' = k^*-1$ (we refer to this fact later).


Next, we provide a lower bound for the cost of \opt.
Since $k^*$ is smallest value for which the FPTAS failed to cover $b$ bins in the $k'$-instance of bin covering, by Lemma~\ref{Jensen}, an optimal covering algorithm \opt cannot cover more than $b/(1-\epsilon_0)$ bins in the $k'$-instance. Let $\epsilon_1 = \epsilon_0/(1-\epsilon_0)$ . So, \opt cannot cover more than $b (1+\epsilon_1)$ bins in the $k'$-instance.
Let $\alpha = (1-\epsilon_2)k'$, where $\epsilon_2 = \frac{4}{1-c^*}\epsilon_1$. We claim that \opt cannot cover $b$ bins in the $\alpha$-instance of the bin covering problem. If this claim is true, then there is no burning scheme that completes within $\alpha-1$ rounds; otherwise, by Lemma \ref{TTT} that burning scheme yields to a covering solution that covers $b$ bins of the $\alpha$-instance of the bin covering problem. In summary, we will have a burning scheme that completes in $k^*$ rounds while an optimal burning algorithm requires $\alpha-1$ rounds to burn the graph. This gives an approximate ratio of $k^* / (\alpha-1)$ which approaches to $\frac{1}{1-\epsilon_2}$ 
$= 1+ \frac{4\epsilon_0}{(1-\epsilon_0)(1-c^*)-4\epsilon_0}=1+\epsilon$
for large values of $k^*$. Note that, since $k^*$ is lower-bounded by the number of paths, we have $k^*\in \omega(1)$.

It remains to show that an optimal covering algorithm cannot cover $b$ bins in the $\alpha$-instance of bin covering. Note that the $\alpha$-instance is similar to the $k'$-instance except that, among the small items, the $\epsilon_2 k'$ largest items are missing.
Call these items \emph{critical items}. We claim that the total size of critical items, denoted by $X$, is at least $2\epsilon_1 b$. 
For now assume it is true; 
by Lemma \ref{optcover}, removing items with total size at least $X$ decreases the number of covered bins in an optimal solution of the $k'$-instance by at least $X/2$. Thus, if it is possible to cover $b$ bins in the $\alpha$-instance 
then it is possible to cover at least $b+ \epsilon_1 b $ bins in the $k'$-instance, which we know is not possible. We conclude that we cannot cover $b$ bins in the $\alpha$-instance. We are just left to show $X > 2\epsilon_1 b$. We have 
$X = k'^2 (2\epsilon_2-\epsilon_2^2)/2C+k'\epsilon_2/2C > k'^2 \epsilon_2/2C $.
Therefore, it suffices to have $k'^2 \epsilon_2/2C > 2\epsilon_1 b$; that is, $\epsilon_2>\frac{4 b c}{k'^2} \epsilon_1$. We previously observed that $\frac{bc}{k'^2} < \frac{1}{1-c^*}$. So, the inequality holds as long as  $\epsilon_2 \geq \frac{4}{1-c^*} \epsilon_1$.
\end{proof}

\darkblue{
\subsubsection*{PTAS for the General Case of Non-Constant Disjoint Paths}
In this section, we use a direct approach to provide a PTAS for graphs formed by non-constant number of disjoint paths. Unlike previous section, we do not make any assumption on the length of the paths, in particular, the length of paths can be asymptotically larger than the number of paths. We note that when \blue{regularity condition holds}, the result of the previous section is stronger as the provided algorithm is \textit{fully} polynomial. 

Assume the graph is formed by $b$ disjoint paths, each of length at most $n$. An instance of the decision variant of the burning problems has a parameter $g$ and asks whether it is possible to burn the graph with fires started at times $1,2 \ldots, g$.
We define the \emph{radius} of a fire started at round $t$ as $g-t+1$; so an instance $I(g)$ of the decision problem asks whether it is possible to burn the graph with fires of radii $1,2, \ldots, g$. Given a constant integer $k$, we form at most $k+1$ groups of fires, each containing fires of close radii such that the difference in the radii of any two fire in a group is at most $\beta = \lfloor g/k \rfloor$ (there will be $\beta$ fires in each group, except potentially the last one). Based on this grouping, we define two new instances of the decision problem: in the \emph{weak instance} $I'(g,k)$, the fires of the first group (with smallest radii) are removed and the radii of fires of other groups is rounded to the smallest radius in the group. In the \emph{strong instance} $I''(g,k)$ the radii are rounded up to the largest radius in the group. Note that in both weak and strong instances, there are $k+1$ radius sizes, and each fire has radius at least $\lfloor g/k \rfloor$ + 1. Also, note that if we remove the $\beta$ fires of largest radii from the strong instance $I''(g,k)$, the result will be the weak instance $I'(g,k)$.
We prove the following lemma, which will be later applied on the weak instances of the problem.

\begin{lemma}\label{lem:ptas}
Consider an instance of the burning problem on disjoint paths in which there are $g$ fires each having a radius among $k+1$ possible radii for some constant $k$ so that each radius is in the range $(\lfloor g/k \rfloor, g]$. There is an algorithm that answers the problem with the following guarantees. If the answer is `yes', then it is possible to burn the graph with the fires in the instances. If the answer is `no', then there is a number $p \in o(g)$ so that it is not possible to burn the graph 
when the $p$ fires of largest radii are removed.
\end{lemma}

\begin{proof}
Assume there are $n$ vertices in the graph. We divide the paths in the graph into short paths with length $O(g)$ and long paths with length $\omega(g)$. If the number of long paths is $\Omega(g)$, the algorithms sends no: there are $\omega(g^2)$ vertices in the graph while the maximum number of vertices that can be burned with the instance is $O(g^2)$. Next, assume the number of long paths is $p$ where $p \in o(g)$. Also, assume the number of paths is at most $g$; otherwise, the algorithm returns `no' as there is no way to burn more than $g$ disjoint paths with $g$ fires. In order to achieve the desired guarantees, we exhaustively check all possible burning schemes for short paths and use a simple strategy to burn long paths.
Assume all short paths have length at most $\alpha g$ for some constant $\alpha$. Since all fires have radius more than $g/k$, it suffices to use at most $\lceil \alpha k/2 \rceil$ fires to burn each path. So, for each path, we have at most $\lceil \alpha k/2 \rceil$ fires each having one of the $k$ possible radii. There are $\tau = \binom{\lceil \alpha k/2 \rceil + k}{k}$ ways to assign fires to each path; define each such assignment a `fire schedule' for a path. Note that $\tau$ is a constant. There are at most $g$ short paths each taking one of the possible $\tau$ fire schedules. So, there are $\binom{g+\tau}{\tau}$ ways to assign fire schedules to these paths; this value is polynomial in $g$ as $\tau$ is constant. We conclude that there is a polynomial number of possible burning schedules for short paths. For each such schedule for the short paths, we complete the burning by using the fires absent in the schedule to burn long paths. We process these fires in an arbitrary order and assign them one by one to the long paths. A fire of radius $r$ burns up to $2r-1$ vertices. When this fire is assigned to a path, $2r-1$ vertices in the path are declared `burned' and the process continues until all vertices in the path are burned, after which the fires are assigned to burn the next path. This process continues until all paths are burned, in which case the algorithm returns `yes'. If we run out of fires and not all paths are burned, the burning schedule for the short paths is not useful and the process continues by checking the next schedule for short paths. If all schedules for short paths are checked and for all of them we fail to burn long paths with the remaining fires, the algorithm returns `no'.

Next, we show the algorithm provides the desired guarantees. First, if the algorithm returns `yes', then there has been a schedule to burn short paths and the remaining fires have successfully burned the long paths. Hence, there is a schedule for fires in the instance that burns the whole graph.
Next, assume the algorithm returns `no'; we claim no algorithm can burn the graph using the same fires when the $p$ fires with the largest radii are removed (recall that $p \in o(g)$ is the number of long paths). Consider otherwise, that is, assume it is possible to burn the graphs with the mentioned fires. The burning schedule for assigning fires to short paths in such solution $S$ is checked also by the algorithm. The difference is that the algorithm assigns fires to long paths differently from $S$. Since the algorithm returns `no', it fails to cover all paths with fires. Hence, if we remove the last fire assigned to each path by the algorithm, the number of vertices that can be burned by the remaining fires will be less than the total size of long paths. Consequently, if we remove the $p$ fires with the largest radii, the remaining fires do not suffice to burn the long paths. In summary, if we assign fires to short paths in the same way that $S$ does and remove the largest $p$ fires from the rest of fires, the remaining fires cannot burn long paths. This contradicts our assumption that $S$ can burn all graphs with the same fires.
\end{proof}

\begin{theorem}
Given a graph of size $n$ formed by a forest of $b = \omega(1)$ disjoint paths and a positive value $\epsilon$, there is an algorithm that generates a burning scheme that completes within a factor $1+ \epsilon$ of an optimal scheme. The time complexity of the algorithm is polynomial in $n$.
\end{theorem}

\begin{proof}
Let $k = \lceil 1/\epsilon \rceil +1$. We exhaustively apply Lemma \ref{lem:ptas} to find the smallest value of $g$ so that the algorithm of the lemma returns `yes' for the weak instance $I'(g,k)$ of the problem. Since the graph can be burned with such weak instance, it can be burned with the actual instance formed by fires of radii $(1,2, \ldots, g)$ (this only involves increasing the radii of fires in the solution provided by the weak instance). So, we can burn the graph in $g$ rounds. Next we provide a lower bound for \opt.

Since the algorithm returns `no' for the weak instance $I'(g-1,k)$, by Lemma \ref{lem:ptas}, it is not possible to burn the graph with fires in the weak-instance in which $p \in o(g)$ largest fires are removed for some value of $p$.
Recall that the weak instance $I'(g-1,k)$ is similar to the strong instance $I''(g-1,k)$ in which $\beta = \lfloor (g-1)/k \rfloor$ fires of largest radii are removed. We conclude that, the strong instance $I''(g-1,k)$ in which $\beta - o(g)$ fires with largest radii are removed cannot burn the graph. Meanwhile, such strong instance is similar to the regular instance formed by fires of radii $1,2, \ldots, g-1-\beta - o(g)$ in which some fires radii is increased. We conclude that it is not possible to burn the graph in $g-1-\beta -o(g) $ rounds. This implies $\opt \geq g(1-1/k) -o(g)$. The ratio between the cost of the algorithm and \opt approaches to $\frac{g}{g(1-1/k)} = 1 + 1/(k-1)$, which is at most $1+\epsilon$.
%
%
\end{proof}
}

\section*{Concluding Remarks}
For general graphs, we provided an approximation algorithm with constant factor of 3. This result shows the burning problem is different from problems such as the Firefighter problem that do not admit constant approximations. The approximation factor is likely to be improved. However, such improvement requires a different (and more involved) argument that improves the lower bounds of Lemma \ref{lem1} for the cost of \opt.

It is not clear whether the burning problem admits a PTAS or is APX-hard for general graphs. A potential APX-hardness proof requires an approach different from the current reductions which are confined to input graphs that are forests of paths. Recall that we showed there is a PTAS for these instances. As the \blue{existing negative results are confined to} sparse, disconnected graphs, and since a PTAS exists for disjoint forests of paths, it might be possible that a PTAS exists for general graphs. We note that the hardness results concerning similar problems such as $k$-center and dominating set problems cannot be applied to show APX-hardness of the burning problem.

\bibliographystyle{plain}
\bibliography{refs}

\end{document}